\documentclass[11pt]{amsart}

\usepackage[T1]{fontenc}
\usepackage{lmodern}
\usepackage{microtype}
\usepackage{amsmath,amssymb,mathtools}
\usepackage[hidelinks]{hyperref}

\newcommand{\DS}{\operatorname{DS}}
\newcommand{\EDS}{\operatorname{EDS}}
\newcommand{\id}{\operatorname{id}}
\newcommand{\Fix}{\operatorname{Fix}}
\newcommand{\Int}{\operatorname{int}}
\newcommand{\coker}{\operatorname{coker}}
\newcommand{\lk}{\operatorname{lk}}
\newcommand{\F}{\mathbb{F}}
\newcommand{\Z}{\mathbb{Z}}
\newcommand{\C}{\mathbb{C}}

\newtheorem{theorem}{Theorem}[section]
\newtheorem{proposition}[theorem]{Proposition}
\newtheorem{lemma}[theorem]{Lemma}
\newtheorem{corollary}[theorem]{Corollary}
\theoremstyle{definition}
\newtheorem{definition}[theorem]{Definition}
\theoremstyle{remark}
\newtheorem{remark}[theorem]{Remark}

\title[Unbounded surgery number gaps]
{Unbounded Gaps Between Ordinary and Equivariant Dehn Surgery Numbers}

\author{Qilong Guo$^{*}$}
\address{College of Science, China University of Petroleum--Beijing, Beijing 102249, P. R. China}
\email{guoqilong1984@hotmail.com}
\thanks{$^{*}$Both authors are corresponding authors.}

\author{Chunxing Yan$^{*}$}
\address{School of Finance and Mathematics, Huainan Normal University, Huainan 232038, Anhui, P. R. China}
\email{cxyan@hnnu.edu.cn}

\date{Version 8, August 5, 2026}

\begin{document}

\begin{abstract}
For a closed oriented $3$-manifold $Y$ and an orientation-preserving involution $\tau$, let $\DS(Y)$ denote the minimum number of components in an integral surgery description of $Y$, and let $\EDS(Y,\tau)$ denote the corresponding minimum among periodic surgery descriptions inducing $\tau$. We prove that for every integer $k\geq 1$ there is a pair $(Y_k,\tau_k)$ such that
\[
   \DS(Y_k)=k,
   \qquad
   \EDS(Y_k,\tau_k)=2k.
\]
Consequently, the difference $\EDS(Y,\tau)-\DS(Y)$ is unbounded even when $\tau$ is an involution. This answers Problems~1.15(b) and~1.15(c) in the K3 problem list. We also construct infinitely many pairwise nonhomeomorphic irreducible lens spaces $Z$ admitting involutions $\sigma$ for which
\[
   \DS(Z)<\EDS(Z,\sigma).
\]
\end{abstract}

\maketitle

\section{Introduction}

By the Lickorish--Wallace theorem, every closed connected oriented $3$-manifold is obtained by integral Dehn surgery on a framed link in $S^3$~\cite{Lickorish,Wallace}. The least possible number of components in such a link is the integral Dehn surgery number $\DS(Y)$. Already the case $\DS(Y)=1$ is subtle. Auckly constructed irreducible homology spheres that are not obtained by surgery on a knot~\cite{Auckly}; further obstructions and examples were found by Hom--Karakurt--Lidman~\cite{HKL} and Liu--Piccirillo~\cite{LP}.

Sakuma's equivariant form of the surgery theorem says that every finite-order orientation-preserving diffeomorphism of a closed oriented $3$-manifold can be realized by integral surgery on a framed link preserved by a standard rotation of $S^3$~\cite{Sakuma}. If the surgery description is required to induce a specified finite-order diffeomorphism $\phi$, the corresponding minimum number of components is the equivariant integral Dehn surgery number $\EDS(Y,\phi)$. Thus $\DS$ is an invariant of $Y$, while $\EDS$ is an invariant of the pair $(Y,\phi)$. Periodic surgery descriptions and their homological consequences were studied by Przytycki--Sokolov~\cite{PS}; equivariant surgery formulas in Floer theory appear in~\cite{Mallick,HMSZ}. Forgetting the symmetry gives
\[
   \DS(Y)\leq \EDS(Y,\phi).
\]

The following two questions from Problem~1.15 in the K3 problem list of Baykur, Kirby, and Ruberman~\cite{K3} motivate this work:
\begin{enumerate}
\item[(b)] Can the difference
\[
   \EDS(Y,\phi)-\DS(Y)
\]
be arbitrarily large?

\item[(c)] Can one have
\[
   \DS(Y)<\EDS(Y,\phi)
\]
when $\phi$ is an involution?
\end{enumerate}
Our main theorem answers both questions affirmatively.

\begin{theorem}\label{thm:main}
For every integer $k\geq 1$, there is a closed connected oriented $3$-manifold $Y_k$ and a smooth orientation-preserving involution $\tau_k\colon Y_k\to Y_k$ such that
\[
   \DS(Y_k)=k,
   \qquad
   \EDS(Y_k,\tau_k)=2k.
\]
Consequently,
\[
   \EDS(Y_k,\tau_k)-\DS(Y_k)=k.
\]
As an unoriented manifold, $Y_k$ may be taken to be the connected sum of $k$ copies of $L(15,4)$.
\end{theorem}

For $k>1$, the manifolds in Theorem~\ref{thm:main} are reducible. Strict inequality also occurs in an infinite irreducible family.

\begin{theorem}\label{thm:irreducible}
There are infinitely many pairwise nonhomeomorphic lens spaces $Z$ admitting orientation-preserving involutions $\sigma$ such that
\[
   \DS(Z)=1<\EDS(Z,\sigma).
\]
\end{theorem}

Section~2 fixes the conventions and records the homological facts used later. Section~3 constructs the basic example, and Section~4 establishes the lower bound for equivariant surgery numbers. The examples with arbitrarily large gaps are obtained in Section~5. Section~6 gives the irreducible lens-space family.

\section{Preliminaries}

All manifolds are smooth and oriented, and all diffeomorphisms preserve orientation unless stated otherwise. We use $\phi$ for a finite-order diffeomorphism in general and reserve $\tau$ for an involution. A link is allowed to have one component, and surgery on the empty link is $S^3$.

Let $L=K_1\cup\cdots\cup K_m\subset S^3$ be an oriented link. Write $\mu_i$ for a meridian of $K_i$ and $\lambda_i$ for its preferred longitude. An integral framing on $K_i$ is the unoriented slope represented by
\[
   n_i\mu_i+\lambda_i,
   \qquad n_i\in\Z.
\]
For an integrally framed link $L=(L;n_1,\dots,n_m)$, let $S^3_L$ denote the resulting oriented manifold.

\begin{definition}
For a closed connected oriented $3$-manifold $Y$, its \emph{integral Dehn surgery number} is
\[
   \DS(Y)=\min\bigl\{\#\pi_0(L):L\subset S^3
   \text{ is integrally framed and }S^3_L\cong Y\bigr\}.
\]
In particular, $\DS(S^3)=0$.
\end{definition}

Reversing orientation does not change $\DS$: if surgery on $L$ with coefficients $n_1,\dots,n_m$ gives $Y$, then surgery on the mirror of $L$ with coefficients $-n_1,\dots,-n_m$ gives $-Y$.

Let $\rho_n\colon S^3\to S^3$ be an orientation-preserving diffeomorphism conjugate to the standard rotation through angle $2\pi/n$ about an unknot. Its fixed set will be called the \emph{rotation axis}. Following Sakuma~\cite{Sakuma}, we use the following terminology.

\begin{definition}
An integrally framed link $L=(L;n_1,\dots,n_m)$ in $S^3$ is \emph{$\rho_n$-periodic} if:
\begin{enumerate}
\item $\rho_n(L)=L$ as a set;
\item every component of $L$ is either disjoint from $\Fix(\rho_n)$ or equal to $\Fix(\rho_n)$;
\item if $\rho_n(K_i)=K_j$, then $n_i=n_j$.
\end{enumerate}
The description is called \emph{augmented} when the rotation axis is a component of $L$, and \emph{non-augmented} otherwise.
\end{definition}

Choose invariant tubular neighborhoods of the components of $L$. The equality of the coefficients on each orbit allows the boundary action to extend over the surgery solid tori, producing a periodic diffeomorphism
\[
   \rho_L\colon S^3_L\longrightarrow S^3_L.
\]
Up to conjugacy by a diffeomorphism isotopic to the identity, this induced map is independent of the choices~\cite[Section~2]{Sakuma}.

\begin{definition}
Let $\phi\colon Y\to Y$ be an orientation-preserving diffeomorphism of order $n$. An \emph{integral periodic surgery description} of $(Y,\phi)$ consists of a $\rho_n$-periodic integrally framed link $L$ and an orientation-preserving diffeomorphism
\[
   h\colon Y\longrightarrow S^3_L
\]
such that
\[
   h\phi h^{-1}=\rho_L.
\]
Thus the chosen generator $\phi$, and not only the cyclic subgroup that it generates, is part of the data.
\end{definition}

By Sakuma's equivariant surgery theorem~\cite{Sakuma}, every orientation-preserving finite-order diffeomorphism of a closed oriented $3$-manifold admits an integral periodic surgery description. In particular, the following invariant is finite.

\begin{definition}
For a pair $(Y,\phi)$ as above, its \emph{equivariant integral Dehn surgery number} is
\[
   \begin{aligned}
   \EDS(Y,\phi)=\min\bigl\{\#\pi_0(L):{}&L\text{ is an integral periodic surgery}\\[-2pt]
      &\text{description of }(Y,\phi)\bigr\}.
   \end{aligned}
\]
\end{definition}

Forgetting the periodic structure gives
\begin{equation}\label{eq:basic-ineq}
   \DS(Y)\leq \EDS(Y,\phi).
\end{equation}

For relatively prime integers $p>0$ and $q$, we use the quotient convention
\[
   L(p,q)=S^3/\langle g\rangle,
   \qquad
   g(z_1,z_2)=(\zeta z_1,\zeta^q z_2),
   \qquad
   \zeta=e^{2\pi i/p}.
\]
Under the identification of the deck group with
\[
   \pi_1(L(p,q))\cong H_1(L(p,q);\Z)\cong \Z/p,
\]
the element $g$ is the preferred generator. Other conventions may replace $q$ by $\pm q^{\pm1}$ modulo $p$.

Moser's theorem supplies the one-knot surgery descriptions needed below. Let $K(r,s)$ denote the $(r,s)$-torus knot, where $r>s>0$ and $\gcd(r,s)=1$. Moser writes a filling slope as $p\lambda-q\mu$~\cite[Section~1, p.~738]{Moser}. To avoid confusing these parameters with those of a lens space, we rename them $(m,n)$.

\begin{theorem}\label{thm:Moser}
Let $r>s>0$ be relatively prime, and let $m,n$ be relatively prime integers with $m>0$. If
\[
   |rsm+n|=1,
\]
then surgery on $K(r,s)$ with filling slope $m\lambda-n\mu$ yields
\[
   L(|n|,ms^2),
\]
where the second parameter is taken modulo $|n|$.
\end{theorem}

\begin{remark}\label{rem:conventions}
In our convention, integer $N$-surgery means filling along $N\mu+\lambda$. Thus Moser's surgery of type $(1,-N)$ is precisely $+N$-surgery in our convention.
\end{remark}

The following elementary facts give lower bounds on the number of components in a surgery description.

\begin{lemma}\label{lem:filling}
Let $X$ be the exterior of an $m$-component link in $S^3$, and let $Y$ be obtained by filling every boundary component of $X$. For every commutative coefficient ring $R$, inclusion induces a surjection
\[
   H_1(X;R)\twoheadrightarrow H_1(Y;R).
\]
When the filling is equivariant, the surjection is equivariant.
\end{lemma}

The surgery linking matrix gives the following standard homology presentation; see, for example,~\cite[Lecture~2]{Saveliev}.

\begin{lemma}\label{lem:matrix}
Suppose that integral surgery on the oriented link $L=K_1\cup\cdots\cup K_m$ with coefficients $n_1,\dots,n_m$ gives $Y$. Let $A=(a_{ij})$ be the symmetric matrix defined by
\[
   a_{ii}=n_i,
   \qquad
   a_{ij}=\lk(K_i,K_j)\quad(i\neq j).
\]
Then
\[
   H_1(Y;\Z)\cong \coker(A\colon\Z^m\to\Z^m).
\]
\end{lemma}

For a finitely generated abelian group $A$, let $d(A)$ denote its minimum number of generators.

\begin{corollary}\label{cor:ordinary-bound}
For every closed oriented $3$-manifold $Y$,
\[
   \DS(Y)\geq d\bigl(H_1(Y;\Z)\bigr).
\]
\end{corollary}

For a one-component periodic surgery description, the homology of the knot exterior gives a useful obstruction.

\begin{lemma}\label{lem:one-component}
If $(Y,\phi)$ has a one-component periodic surgery description, then
\[
   \phi_*=\pm\id
   \qquad\text{on }H_1(Y;\Z).
\]
This applies to both augmented and non-augmented descriptions.
\end{lemma}

\begin{proof}
Let $K\subset S^3$ be the periodic framed knot and let $X=S^3\setminus\Int N(K)$. Since $H_1(X;\Z)\cong\Z$, the induced automorphism of $H_1(X;\Z)$ is multiplication by $+1$ or $-1$. Lemma~\ref{lem:filling} gives an equivariant quotient map
\[
   H_1(X;\Z)\twoheadrightarrow H_1(Y;\Z),
\]
so the induced action on the quotient is again $+\id$ or $-\id$.
\end{proof}

\begin{corollary}\label{cor:two-components}
If $\phi_*\neq\pm\id$ on $H_1(Y;\Z)$, then
\[
   \EDS(Y,\phi)\geq 2.
\]
\end{corollary}

\section{The symmetric Hopf-link block}

Consider
\[
   S^3=\bigl\{(z_1,z_2)\in\C^2:|z_1|^2+|z_2|^2=1\bigr\}
\]
and set
\[
   K_1=\{z_2=0\},
   \qquad
   K_2=\{z_1=0\}.
\]
Then $K_1\cup K_2$ is a Hopf link. Consider the involution
\[
   \rho(z_1,z_2)=(z_2,z_1).
\]
This map is unitary and therefore orientation-preserving. In the unitary coordinates
\[
   u=\frac{z_1+z_2}{\sqrt2},
   \qquad
   v=\frac{z_1-z_2}{\sqrt2},
\]
we have $\rho(u,v)=(u,-v)$. Thus $\rho$ is conjugate to the standard half-turn about the unknot $\{v=0\}$. It exchanges $K_1$ and $K_2$, both of which are disjoint from the rotation axis.

Give both components framing $+4$. The resulting framed link is $\rho$-periodic. Let $Y$ denote the surgered manifold, and let $\tau\colon Y\to Y$ be the involution induced by $\rho$. The diagram has two components, so
\[
   \EDS(Y,\tau)\leq 2.
\]

\begin{proposition}\label{prop:block}
There is an identification
\[
   H_1(Y;\Z)\cong\Z/15
\]
under which $\tau_*$ is multiplication by $-4$. Moreover,
\[
   \DS(Y)=1,
   \qquad
   \EDS(Y,\tau)=2.
\]
\end{proposition}

\begin{proof}
Orient $K_1$ and $K_2$ by
\[
   t\longmapsto(e^{it},0),
   \qquad
   t\longmapsto(0,e^{it}).
\]
They form the positive Hopf link, so $\lk(K_1,K_2)=1$. The involution $\rho$ preserves these orientations while exchanging the components. The surgery linking matrix is therefore
\[
   A=\begin{pmatrix}4&1\\[2pt]1&4\end{pmatrix}.
\]
If $x$ and $y$ denote the meridian classes, Lemma~\ref{lem:matrix} gives
\[
   H_1(Y;\Z)
   \cong
   \langle x,y\mid 4x+y=0,\ x+4y=0\rangle.
\]
The first relation gives $y=-4x$, and the second then gives $15x=0$. Since $|\det A|=15$, the element $x$ has order exactly $15$. Hence $H_1(Y;\Z)\cong\Z/15$. Because the half-turn exchanges the two positively oriented meridians,
\[
   \tau_*(x)=y=-4x.
\]
Since $-4\not\equiv\pm1\pmod{15}$, Corollary~\ref{cor:two-components} gives $\EDS(Y,\tau)\geq2$. The opposite inequality is supplied by the Hopf-link diagram, and hence $\EDS(Y,\tau)=2$.

To compute the ordinary surgery number, first apply a slam-dunk move~\cite[Section~5.3]{GS}. Eliminating one component changes the coefficient on the other to
\[
   4-\frac14=\frac{15}{4}.
\]
Thus, as an unoriented manifold, $Y$ is $L(15,4)$. The same lens space is obtained by applying Theorem~\ref{thm:Moser} to $K(7,2)$ with $(m,n)=(1,-15)$. By Remark~\ref{rem:conventions}, this is $+15$ integral surgery in our convention, and
\[
   |7\cdot2\cdot1-15|=1.
\]
Moser's theorem therefore gives $L(15,4)$. If its orientation is opposite to that of $Y$, mirroring the knot and reversing the coefficient gives the required orientation. Hence $\DS(Y)\leq1$. Since $H_1(Y;\Z)$ is nontrivial, $Y\neq S^3$, so $\DS(Y)=1$.
\end{proof}

\section{A homological lower bound}

The action computed in the preceding section has the additional property
\begin{equation}\label{eq:congruences}
   -4\equiv 1\pmod5,
   \qquad
   -4\equiv -1\pmod3.
\end{equation}
Thus $\tau_*$ acts as $+\id$ on $H_1(Y;\F_5)$ and as $-\id$ on $H_1(Y;\F_3)$. This suggests measuring the positive and negative eigenspaces over different odd primes. We now show that the resulting dimensions can nevertheless be added to give a lower bound for the equivariant surgery number.

For an odd prime $\ell$ and $\delta\in\{+1,-1\}$, define
\[
   b^\delta_\ell(Y,\tau)
   =\dim_{\F_\ell}
   \ker\bigl(\tau_*-\delta\id:H_1(Y;\F_\ell)\to H_1(Y;\F_\ell)\bigr),
\]
and set
\[
   b^\delta(Y,\tau)=\max_{\ell\text{ odd prime}}b^\delta_\ell(Y,\tau).
\]
These maxima are finite because $H_1(Y;\Z)$ is finitely generated. We restrict to odd primes so that $+1$ and $-1$ are distinct and the projections $(1\pm\tau_*)/2$ are defined.

To compare eigenspace dimensions over different coefficient fields, we shall use the following elementary linear-algebraic fact about signed permutation involutions.

\begin{lemma}\label{lem:signed-permutation}
Let $V\cong\Z^m$ have basis $e_1,\dots,e_m$, and let $T\colon V\to V$ be an involution satisfying
\[
   T(e_i)=\varepsilon_i e_{\pi(i)},
   \qquad
   \varepsilon_i\in\{\pm1\}.
\]
There are nonnegative integers $r_+$ and $r_-$ with $r_++r_-=m$ such that, over every field $\Bbbk$ of characteristic different from $2$,
\[
   \dim_{\Bbbk}\ker(T-\id)=r_+,
   \qquad
   \dim_{\Bbbk}\ker(T+\id)=r_-.
\]
In particular, these dimensions do not depend on the field.
\end{lemma}

\begin{proposition}\label{prop:homological-bound}
For every closed oriented $3$-manifold $Y$ with an orientation-preserving involution $\tau$,
\[
   \EDS(Y,\tau)\geq b^+(Y,\tau)+b^-(Y,\tau).
\]
\end{proposition}

\begin{proof}
Let $L$ be an $m$-component integral periodic surgery description inducing $(Y,\tau)$, and let
\[
   X=S^3\setminus\Int N(L)
\]
be the link exterior. Choose orientations on the components of $L$. Their meridians form a basis of $H_1(X;\Z)\cong\Z^m$. The ambient half-turn permutes the components and sends each oriented meridian to a meridian or its negative. Its action $T$ on $H_1(X;\Z)$ is therefore a signed permutation involution.

By Lemma~\ref{lem:signed-permutation}, there are integers $r_+,r_-\geq0$, independent of the odd prime $\ell$, such that
\[
   \begin{aligned}
   H_1(X;\F_\ell)&=W^+_\ell\oplus W^-_\ell,\\
   \dim W^+_\ell&=r_+,
   \qquad
   \dim W^-_\ell=r_-,
   \qquad
   r_++r_-=m.
   \end{aligned}
\]
Lemma~\ref{lem:filling} gives an equivariant surjection
\[
   f_\ell\colon H_1(X;\F_\ell)\twoheadrightarrow H_1(Y;\F_\ell),
   \qquad
   f_\ell T=\tau_*f_\ell.
\]
The restriction of $f_\ell$ to each eigenspace is surjective. To see this, let $y$ be a $\delta$-eigenvector and choose $x$ with $f_\ell(x)=y$. Then
\[
   x_\delta=\frac{x+\delta T(x)}2
\]
is a $\delta$-eigenvector and satisfies $f_\ell(x_\delta)=y$. It follows that
\[
   b^+_\ell(Y,\tau)\leq r_+,
   \qquad
   b^-_\ell(Y,\tau)\leq r_-
\]
for every odd prime $\ell$. Since $r_+$ and $r_-$ are independent of $\ell$, the two maxima may be taken separately:
\[
   b^+(Y,\tau)\leq r_+,
   \qquad
   b^-(Y,\tau)\leq r_-.
\]
It follows that
\[
   b^+(Y,\tau)+b^-(Y,\tau)\leq r_++r_-=m.
\]
Minimizing over all periodic surgery descriptions proves the proposition.
\end{proof}

\begin{remark}\label{rem:different-primes}
The maxima defining $b^+$ and $b^-$ need not occur at the same prime. Lemma~\ref{lem:signed-permutation} is precisely what permits this: both quantities are bounded by field-independent dimensions of the meridian representation. For the block in Section~3, the positive contribution is detected modulo $5$ and the negative contribution modulo $3$.
\end{remark}

\section{Arbitrarily large gaps}

We now construct the pair $(Y_k,\tau_k)$ appearing in Theorem~\ref{thm:main}. The construction starts by placing $k$ copies of the basic block under a single half-turn. Choose a point $p_\infty$ on the rotation axis, disjoint from the Hopf link, and identify $S^3\setminus\{p_\infty\}$ with $\mathbb R^3$ by stereographic projection. In suitable coordinates, the involution becomes
\[
   R(x,y,z)=(-x,-y,z),
\]
the half-turn about the $z$-axis. Let $H$ denote the $(+4,+4)$-framed Hopf link from Section~3. Since $H$ is compact, it is contained in an $R$-invariant round ball $B$. Write the center and radius of $B$ as $(0,0,c)$ and $r$, respectively.

For $i=1,\dots,k$, define
\[
   s_i(x,y,z)=(x,y,z+3ri)
\]
and set
\[
   B_i=s_i(B),
   \qquad
   H_i=s_i(H).
\]
The ball $B_i$ has center $(0,0,c+3ri)$ and radius $r$. If $i\neq j$, the distance between the centers of $B_i$ and $B_j$ is $3r|i-j|\geq3r>2r$, so the balls are pairwise disjoint. Moreover, each $s_i$ commutes with $R$, so every $B_i$ is $R$-invariant and $R$ exchanges the two components of $H_i$. Set
\[
   L_k=H_1\cup\cdots\cup H_k.
\]
Then $L_k$ is a split $R$-periodic framed link with $2k$ components. Let $Y_k=S^3_{L_k}$, and let
\[
   \tau_k\colon Y_k\longrightarrow Y_k
\]
be the involution obtained by extending $R$ over the surgery solid tori.

The split-link description gives an orientation-preserving diffeomorphism
\[
   Y_k\cong\mathop{\#}_{i=1}^kY.
\]
Indeed, perform the prescribed surgery on $H_i$ inside $B_i$, leaving $\partial B_i$ fixed, and denote the resulting manifold with boundary by $M_i$. Gluing the untouched complementary ball to $M_i$ gives the result of surgery on a single basic block, so
\[
   M_i\cong Y\setminus\Int(B^3).
\]
The complement of the interiors of the balls $B_1,\dots,B_k$ is a $3$-sphere with $k$ open balls removed. Gluing the manifolds $M_i$ to its boundary components is therefore the standard connected-sum construction.

The connected-sum decomposition is compatible with the involutions. The half-turn $R$ preserves each $B_i$, rather than permuting the balls, and its restriction to $(B_i,H_i)$ is conjugate via $s_i$ to its restriction to the original block $(B,H)$. Thus $\tau_k$ acts on each connected-sum factor as a copy of $\tau$. Proposition~\ref{prop:block} consequently gives
\[
   H_1(Y_k;\Z)
   \cong\bigoplus_{i=1}^kH_1(Y;\Z)
   \cong(\Z/15)^k
\]
and, under this direct-sum identification,
\[
   (\tau_k)_*
   =\bigoplus_{i=1}^k\tau_*
   =(-4)\id.
\]

\begin{proof}[Proof of Theorem~\ref{thm:main}]
The periodic surgery link $L_k$ has $2k$ components, so
\[
   \EDS(Y_k,\tau_k)\leq2k.
\]
For the reverse inequality, reduce the action above modulo $5$ and modulo $3$. By~\eqref{eq:congruences},
\[
   H_1(Y_k;\F_5)\cong\F_5^k,
   \qquad
   (\tau_k)_*=\id,
\]
whereas
\[
   H_1(Y_k;\F_3)\cong\F_3^k,
   \qquad
   (\tau_k)_*=-\id.
\]
Hence
\[
   b^+(Y_k,\tau_k)\geq k,
   \qquad
   b^-(Y_k,\tau_k)\geq k.
\]
Proposition~\ref{prop:homological-bound} gives $\EDS(Y_k,\tau_k)\geq2k$, and therefore
\[
   \EDS(Y_k,\tau_k)=2k.
\]

For the ordinary surgery number, Proposition~\ref{prop:block} gives a one-knot integral surgery description of each summand $Y$. Their split union shows that $\DS(Y_k)\leq k$. Conversely,
\[
   d\bigl(H_1(Y_k;\Z)\bigr)
   =d\bigl((\Z/15)^k\bigr)
   =k,
\]
and Corollary~\ref{cor:ordinary-bound} gives $\DS(Y_k)\geq k$. Thus $\DS(Y_k)=k$.
\end{proof}

\section{An infinite irreducible family}

The reducible examples above give an unbounded gap. To obtain irreducible examples, we use the coordinate-exchange symmetry in the quotient model of a lens space.

\begin{lemma}\label{lem:coordinate-exchange}
Let $p>1$ and let $q$ be relatively prime to $p$. If
\[
   q^2\equiv1\pmod p,
\]
then the coordinate exchange
\[
   \sigma(z_1,z_2)=(z_2,z_1)
\]
descends to an orientation-preserving involution of $L(p,q)$. Under the identification
\[
   \pi_1(L(p,q))\cong H_1(L(p,q);\Z)\cong\Z/p
\]
coming from the deck group, the induced automorphism is multiplication by $q$.
\end{lemma}

\begin{proof}
Let
\[
   g(z_1,z_2)=(\zeta z_1,\zeta^qz_2)
\]
generate the deck group. A direct calculation gives
\[
   \sigma g\sigma^{-1}(z_1,z_2)=(\zeta^qz_1,\zeta z_2).
\]
Since $q^2\equiv1\pmod p$, the right-hand side is $g^q(z_1,z_2)$. Thus $\sigma$ normalizes the deck group and descends to the quotient. Conjugation sends $g$ to $g^q$, so the induced action on the cyclic fundamental group, and hence on first homology, is multiplication by $q$.
\end{proof}

For $a\geq3$, define
\[
   r_a=a^2-a+1,
   \qquad
   p_a=ar_a-1=(a-1)(a^2+1),
   \qquad
   q_a=a^2.
\]
Then $r_a>a$, $\gcd(r_a,a)=1$, and $\gcd(p_a,q_a)=1$.

\begin{proof}[Proof of Theorem~\ref{thm:irreducible}]
Set
\[
   Z_a=L(p_a,q_a).
\]
Apply Theorem~\ref{thm:Moser} to $K(r_a,a)$ with
\[
   (m,n)=(1,1-ar_a)=(1,-p_a).
\]
By Remark~\ref{rem:conventions}, this is $+p_a$ integral surgery in our convention. Since
\[
   |ar_am+n|=1,
\]
the resulting manifold is $L(p_a,a^2)=L(p_a,q_a)$ in Moser's convention. Up to orientation, this is $Z_a$ in the quotient convention; the opposite orientation is handled by mirroring the knot and reversing the surgery coefficient. Hence $\DS(Z_a)\leq1$. Since $H_1(Z_a;\Z)\cong\Z/p_a$ is nontrivial, $\DS(Z_a)=1$.

For the equivariant surgery number, note that
\[
   q_a^2-1=a^4-1=(a+1)(a-1)(a^2+1)=(a+1)p_a.
\]
Thus $q_a^2\equiv1\pmod{p_a}$, and Lemma~\ref{lem:coordinate-exchange} gives an orientation-preserving involution $\sigma_a$ of $Z_a$ whose action on first homology is multiplication by $q_a$. For $a\geq3$,
\[
   1<q_a<p_a-1,
\]
so this automorphism is neither $+\id$ nor $-\id$. Corollary~\ref{cor:two-components} therefore gives
\[
   \EDS(Z_a,\sigma_a)\geq2>1=\DS(Z_a).
\]

The integer $p_a=(a-1)(a^2+1)$ is strictly increasing for $a\geq3$. Hence the groups $H_1(Z_a;\Z)$ have distinct orders, and the lens spaces $Z_a$ are pairwise nonhomeomorphic. Since every lens space is irreducible, the proof is complete.
\end{proof}

\end{document}